\documentclass[11pt]{article}%
\usepackage{amssymb}
\usepackage{amsmath}
\usepackage{amsfonts}
\usepackage{graphicx}%
\providecommand{\U}[1]{\protect\rule{.1in}{.1in}}
\providecommand{\U}[1]{\protect\rule{.1in}{.1in}}
\newtheorem{theorem}{Theorem}[section]

\newtheorem{corollary}[theorem]{Corollary}

\newtheorem{definition}[theorem]{Definition}
\newtheorem{example}[theorem]{Example}

\newtheorem{lemma}[theorem]{Lemma}

\newtheorem{proposition}[theorem]{Proposition}

\newenvironment{proof}[1][Proof]{\noindent\textbf{#1.} }{\ \rule{0.5em}{0.5em}}
\begin{document}

\author{Yulia Kempner\\Department of Computer Science\\Holon Institute of Technology, Israel\\yuliak@hit.ac.il
\and Vadim E. Levit\\Department of Computer Science \\Ariel University, Israel\\levitv@ariel.ac.il}
\date{}
\title{Violator spaces vs closure spaces }
\maketitle

\begin{abstract}
Violator Spaces were introduced by J. Matou\v{s}ek et al. in 2008 as
generalization of Linear Programming problems. Convex geometries were invented
by Edelman and Jamison in 1985 as proper combinatorial abstractions of
convexity. Convex geometries are defined by anti-exchange closure operators.
We investigate an interrelations between violator spaces and closure spaces
and show that violator mapping may be defined by a week version of closure
operators. Moreover, we prove that violator spaces with an unique basis
satisfies the anti-exchange and the Krein-Milman properties.


\end{abstract}

\section{Preliminaries}

The main goal of this paper is to make connections between two well -- but up
to now independently -- developed theories, the theory of violator spaces and
the theory of closure spaces.

LP-type problems have been introduced and analyzed by Matou\v{s}ek, Sharir and
Welzl \cite{MSW},\cite{SW} as a combinatorial framework that encompasses
linear programming and other geometric optimization problems. J. Matou\v{s}ek
et al. define a simpler framework: violator spaces, which constitute a proper
generalization of LP-type problems. Originally, violator spaces were defined
for set of constraints $H$, where with each subset of constraints $G \subseteq
H$ associates $V(G)$ - the set of all constraints violating G.

The classic example of an LP-type problem is the problem of computing the
smallest enclosing ball of a finite set of points in $R^{d}$. Here the set $H$
is a set of points in $R^{d}$, and the violated constraints of some subset of
the points $G$ are exactly the points lying outside the smallest enclosing
ball of $G$.

\begin{definition}
\cite{VS} A \textit{violator space} is a pair $(H,V)$, where $H$ is a finite
set and $V$ is a mapping $2^{H}\rightarrow2^{H}$ such that

Consistency: $G \cap V(G) = \emptyset$ holds for all $G \subseteq H$

Locality: For all $F\subseteq G \subseteq H$, where $G \cap V(F) = \emptyset$,
we have $V(G)=V(F)$.
\end{definition}

Convex geometries were invented by Edelman and Jamison in 1985 as proper
combinatorial abstractions of convexity. There are various ways to
characterize finite convex geometries. One of them defines convex sets by
anti-exchange closure operators. The convex hull operator on Euclidean space
$E^{n}$ is a classic example of a closure operator with anti-exchange property.

In this paper we consider the connection between the mapping $V$ of violator
spaces and closure operators. We show that the mapping $V$ may be defined by
week version of closure operator. Interrelations between violator spaces and
closure spaces gives a new insight on well known results in two theories.

In the paper we consider only finite sets. We will use $X\cup x$ for
$X\cup\{x\}$, and $X-x$ for $X-\{x\}$.

\begin{definition}
We say that $(E,\tau)$ is a \textit{closure space} if $\tau:2^{E}%
\rightarrow2^{E}$ is a \textit{closure operator } satisfying the closure axioms:

C1: $X \subseteq\tau(X)$

C2: $X \subseteq Y\Rightarrow\tau(X)\subseteq\tau(Y)$

C3: $\tau(\tau(X))=\tau(X)$.
\end{definition}

\begin{definition}
A set $A \subseteq E$ is a \textit{closed set } if $A = \tau(A)$.
\end{definition}

The family of closed sets $K=\{X \in E: X=\tau(X)\}$ is closed under
intersection. Indeed,

$X,Y \in K \Rightarrow X \cap Y \subseteq\tau(X \cap Y)\subseteq\tau(X)
\cap\tau(Y)= X \cap Y$.

Conversely, any set system (E,K) closed under intersection is a family of
closed sets of the closure operator ~$\tau_{K}(X)=\cap\{A \in K: X \subseteq A
\}$.

$(E,\tau)$ is a \textit{convex geometry} if it satisfies the
\textit{anti-exchange axiom}:
\[
p,q \notin\tau(X) \wedge p \in\tau(X \cup{q}) \rightarrow q \notin\tau(X
\cup{p})
\]

A closure space is \textit{unique generated}, if every closed set $X$ has a unique \textit{basis} - a minimal subset $B \subseteq X$ with closure $X=\tau(B)$. A well-known  characterization \cite{Greedoids, CL} of closure operators states equivalence between uniqueness of the basis, anti-exchange property and the Krein-Milman properties.
One of our main findings is that violator spaces have the same property.

The rest of our paper is organized as follows. In Section 2, we investigate
interrelations between violator and closure spaces. We prove that every
closure space is a violator space, describe the violator mapping as a week
closure operator, and give a definition of violator space in terms of closure
space. Based on subsequent weakening of a closure operator we introduce the new notion - \textit{convex space}.
Section 3 is devoted to violator spaces with an unique basis, and
expands the known theorem connecting between uniqueness of the basis with
anti-exchange property to violator spaces. In Section 4 we focus on the role
of extreme points - an important geometric aspect of convex sets. We prove
that uniquely generated violator spaces satisfy the Krein-Milman property and find the conditions for which this theorem holds for convex spaces too.

\section{Violator mapping and closure operator}

\begin{proposition}
\label{cltov}  Let $(E,\tau)$ be a closure space. Define $V(X)=E-\tau(X)$.
Then $(E,V)$ is a violator space.
\end{proposition}

\begin{proof}
Consistency: $G\subseteq\tau(G)$ (from C1), then $G \cap(E-\tau(G))=\emptyset$
, i.e., $G \cap V(G) = \emptyset$.

Locality: For all $F\subseteq G \subseteq E$,
\[
G \cap V(F) = \emptyset\Leftrightarrow G \cap(E-\tau(F)) = \emptyset
\Leftrightarrow G \subseteq\tau(F)
\]

From C2 and C3: $G \subseteq\tau(F) \Rightarrow\tau(G)\subseteq\tau
(\tau(F))=\tau(F)$.

From another side: $F \subseteq G \Rightarrow\tau(F)\subseteq\tau(G)$ (based
on C2).

Thus, $\tau(G)=\tau(F) \Leftrightarrow V(G)=V(F)$.
\end{proof}

What about the opposite direction?

\begin{proposition}
\label{vtocl} Let $(H,V)$ be a violator space. Define $\tau(X)=H-V(X)$. Then
the operator $\tau$ satisfies two closure axioms: C1 and C3.
\end{proposition}

\begin{proof}
Consistency is equivalent to C1:
\[
X \cap V(X) = \emptyset\Leftrightarrow X \cap(H-\tau(X)) = \emptyset
\Leftrightarrow X \subseteq\tau(X)
\]
Prove C3: Since $X \subseteq\tau(X)$ and $\tau(X) \cap V(X) = \emptyset$,
then, from locality, we have $V(\tau(X))=V(X) \Leftrightarrow\tau
(\tau(X))=\tau(X)$.
\end{proof}

There is an example of violator spaces (see \cite{VS}, p.2130) where $F
\subseteq G$ and $V(G)$ is not a subset of $V(F)$, i.e., the axiom C2 is not
hold. Another simple example is as follows.

\begin{example}
\label{ex1} Let $H=\{1,2,3\}$. Define $V(X)=H-X$ for each $X \subseteq H$
except $V(\{1\})=\{2\}$. It's easy to check that $(H,V)$ is a
violator space, but while $\{1\} \subseteq\{1,2\}$, $V(\{1,2\})=\{3\}
\nsubseteq V(\{1\})=\{2\}$.
\end{example}

Note, that the locality of violator spaces is equivalent to

$C22: (F \subseteq G \subseteq\tau(F)) \Rightarrow\tau(G)=\tau(F). $

Thus we have equivalent definition of violator spaces.

\begin{definition}
We say that $(H,\tau)$ is a violator space if $\tau:2^{H}\rightarrow2^{H}$
satisfies the axioms: C1,C22.
\end{definition}

Consider the relation between axioms.

\begin{proposition}
The axioms C2 and C3 implies C22.
\end{proposition}

The proposition follows from Proposition \ref{cltov}, but here we give another
proof, that doesn't use the definition of violator spaces.

\begin{proof}
Let $X \subseteq Y \subseteq\tau(X) $. Then, from C2, $\tau(X) \subseteq
\tau(Y)$. From another side, by C2 and C3:

$Y \subseteq\tau(X) \Rightarrow\tau(Y) \subseteq\tau(\tau(X))=\tau(X)$

Thus, $\tau(X) = \tau(Y)$ .
\end{proof}

\begin{proposition}
The axioms C1 and C22 implies C3.
\end{proposition}

The proof is identical to proof of Proposition \ref{vtocl}:

\begin{proof}
Since, from C1, $X \subseteq\tau(X) \subseteq\tau(X)$, C22 implies $\tau(\tau
(X))=\tau(X)$
\end{proof}

Thus we can see that while closure spaces satisfy the closure axioms C1, C2, and C3, violator spaces  satisfy the axioms C1, C22 and C3,
and so may be considered as \textit{week} closure spaces.

Any violator space $(H,V)$ satisfies \textit{monotonicity} (\cite{VS}, Lemma
17) defined as follows:

$V(F)=V(G) \Rightarrow V(E)=V(F)=V(G)$ for all sets $F \subseteq E \subseteq G
\subseteq H $

Monotonicity immediately follows from consistency and locality.

Rewrite the definition of monotonicity in terms of operator $\tau$:
\begin{equation}
Monotonicity: X \subseteq Y \subseteq Z \wedge\tau(X)=\tau(Z) \Rightarrow
\tau(Y)=\tau(X) = \tau(Z) \label{convex}
\end{equation}

 Since the property holds for each set lying between two sets, in the future we, following \cite{Monjardet}, call the operator satisfying the property (\ref{convex}) \textit{convex operator}.

Consider the relationship between axiom C22 and convexity. Axioms C1 and
C22 imply convexity as follows from the proof in (\cite{VS}, Lemma 17).
Indeed, the axiom C1 yields $Y \subseteq Z \Rightarrow Y \subseteq\tau
(Z)=\tau(X)$. Then we have from C22: $X \subseteq Y \subseteq\tau(X)
\Rightarrow\tau(Y)=\tau(X) = \tau(Z)$.

\begin{proposition}
Convexity and axiom C3 imply C22.
\end{proposition}

\begin{proof}
If $F \subseteq G \subseteq\tau(F)$, then convexity with C3 ($\tau
(F)=\tau(\tau(F)$)implies $\tau(G)=\tau(F)$.
\end{proof}

So we can give another equivalent definition of violator spaces:

\begin{definition}
$(H,\tau)$ is a violator space if $\tau:2^{H}\rightarrow2^{H}$ satisfies the
convexity and the closure axioms: C1 and C3.
\end{definition}

The following example shows that the convexity property with axiom C1 do
not obligate the space to be violator spaces.

\begin{example}
Let $H=\{1,2,3\}$. Define $\tau(X)=X$ for each $X \subseteq H$ except
$\tau(\{1\})=\{1,2\}$, and $\tau(\{1,2\})=\{1,2,3\}$. It's easy to check that
$(H,\tau)$ satisfies the properties C1 and convexity, but while $\{1\}
\subseteq\{1,2\} \subseteq\tau(\{1\})$, $\tau(\{1,2\})=\{1,2,3\}\neq
\tau(\{1\})$ and $\tau(\tau(\{1\})) \neq\tau(\{1\})$.
\end{example}

There is the space satisfying C1
and the convexity, but it doesn't satisfy C22 (and not C3), and so it is not a violator space.
We call such space the \textit{convex space}.
\begin{definition}
$(H,\tau)$ is an convex space if $\tau:2^{H}\rightarrow2^{H}$ satisfies the
convexity and axiom C1.
\end{definition}

\section{Uniquely generated spaces and anti-exchange property}

Here and in the future we will suppose that we have a finite space $(E,\tau)$ - the pair of set $E$ and
operator $\tau:2^{E}\rightarrow2^{E}$.

\begin{definition}
We say that $B \subseteq E$ is a \textit{generator} (known also as a spanning
set) of $X \subseteq E$ if $\tau(B)=\tau(X)$. For $X\subseteq E$, a \textit{
minimal generator or basis of X} is a minimal subset $B$ with $\tau
(B)=\tau(X)$.
\end{definition}

\begin{definition}
A space $(E,\tau)$ is \textit{uniquely generated} if every set $X\subseteq E$
has a unique basis.
\end{definition}

Note that we do not demand from generators, and so from bases, of any set do
be a subset of the set. The situation is changed when a space is uniquely generated.

\begin{proposition}
\label{U1}  A space $(E,\tau)$ is uniquely generated if and only if a basis of
each set $X$ is contained in all generators of $X$.
\end{proposition}

\begin{proof}
Let $B$ be a basis of $X$. Then we have to prove that $(E,\tau)$ is uniquely
generated if and only if
\begin{equation}
B \subseteq\bigcap\{Y: \tau(Y)= \tau(X)\}\label{UQA}%
\end{equation}
1. $\Rightarrow$ Let $Y \subseteq E$, $\tau(X)=\tau(Y)$ and $B$ is a basis of
$X$. $X$ is a generator of $\tau(X)$. If $X$ is not a basis of $\tau(X)$ ,
then there is a minimal set $B_{X}$ contained in $X $ such that $\tau
(B_{X})=\tau(X)$. By analogy, there is a minimal subset $B_{Y} \subseteq Y $
such that $\tau(B_{Y})=\tau(Y)=\tau(X)$. Since the space is a uniquely
generated, $B_{X}=B_{Y}=B$, and so $B \subseteq Y$. The proof is correct for each
generator of $X$, so the inclusion (\ref{UQA}) holds for any uniquely
generated space.

2. $\Leftarrow$ Suppose there are two bases $B_{1} \neq B_{2}$ of a set $X$.
Then $\tau(X)=\tau(B_{2})$, and so from (\ref{UQA}) $B_{1}\subseteq B_{2}$. By
analogy, $B_{2}\subseteq B_{1}$. Thus, $B_{1} = B_{2}$.
\end{proof}

Since each set is a generator of itself, we have the following property.

\begin{corollary}
If $(E,\tau)$ is uniquely generated then each basis $B$ of $X\subseteq E$ is a
subset of $X$.
\end{corollary}

To characterize an uniquely generated violator space we will use the unique
generation property from \cite{CL}.

\begin{proposition}
\label{UQP} An convex space $(E,\tau)$ is uniquely generated if and only if
for every $X,Y\subseteq E$
\begin{equation}
\tau(X)=\tau(Y)\Rightarrow\tau(X\cap Y)=\tau(X)=\tau(Y)\label{UQ}%
\end{equation}

\end{proposition}

\begin{proof}
1. Let a convex space $(E,\tau)$ be uniquely generated. Then the Proposition
\ref{U1} implies that the basis $B$ of $X$ is a subset of $X \cap Y$. Then
(from convexity) $\tau(X \cap Y)=\tau(B) = \tau(X)$.

2. Suppose that there are two bases $B_{1} \neq B_{2}$ of a set $X$. Then by
(\ref{UQ}) $\tau(B_{1} \cap B_{2})=\tau(B_{1})$ in contradiction with
minimality of $B_{1}$.
\end{proof}

We can rewrite the property (\ref{UQ}) as follows: for every set $X \subseteq
E$ of uniquely generated convex space $(E,\tau)$, the basis $B$ of $X$ is an
intersection of all generators of $X$:
\begin{equation}
B = \bigcap\{Y: \tau(Y)= \tau(X)\}\label{UQI}%
\end{equation}
The future elaboration (development) of this formula will be shown (done) in
the next section.

An Example \ref{ex1} may be considered as an example of violator space that is
not uniquely generated ($\tau(\{1\}=\tau(\{3\})=\{1,3\}$). It is easy to see
that here the equation (\ref{UQI}) does not hold, and a basis $\{1\}$ of
$\{3\}$ is not contained in $\{3\}$.

It is known that a closure operator is uniquely generated if and only if it
satisfies the anti-exchange property (\cite{Edelman, Greedoids, CL}. We extend
this characterization to violator spaces.

\begin{theorem}
Let $(E,\tau)$ be a violator space. Then $(E,\tau)$ is uniquely generated if
and only if the operator $\tau$ satisfies the anti-exchange property.
\end{theorem}

At first we prove the following lemma:

\begin{lemma}
Let $(E,\tau)$ be a violator space. Then for each $A \subseteq E$ holds:
\[
x \notin\tau(A) \Leftrightarrow\tau(A) \neq\tau(A \cup x)
\]
\label{Lemma}
\end{lemma}

\begin{proof}
1. $\Rightarrow$ If $x \notin\tau(A)$, but, from C1,$x \in\tau(A \cup x)$,
then $\tau(A) \neq\tau(A \cup x)$. 2. $\Leftarrow$ If $x \in\tau(A)$, then $A
\subseteq A\cup x \subseteq\tau(A)$. Hence, from C22, $\tau(A) = \tau(A
\cup x)$.
\end{proof}

Now, prove the Theorem.

\begin{proof}
1. Unique generation implies anti-exchange property. Suppose there are $p,q
\notin\tau(X)$ with $p \in\tau(X \cup q)$ and $q \in\tau(X \cup p)$. Then (by
using C1) $X \cup p \subseteq X \cup p \cup q \subseteq\tau(X \cup p)$. Then
C22 yields $\tau(X \cup p \cup q)=\tau(X \cup p)$. By analogy, we have $\tau(X
\cup p \cup q)=\tau(X \cup q)$. Then from Proposition \ref{UQP} $\tau
(X)=\tau(X \cup p)$ implying $p \in\tau(X)$, a contradiction.

2. Anti-exchange property implies unique generation. Let $\tau(X)=\tau(Y)$,
and let $B_{X}$ be a minimal set contained in $X $ such that $\tau(B_{X}%
)=\tau(X)$. To prove that the space is uniquely generated enough to prove (by
the Proposition \ref{U1}) that $B_{X} \subseteq Y$. Suppose there are $p \in
B_{X}$ and $p \notin Y$. Since $B_{X}$ is a minimal generator (basis) of $X$,
$\tau(B_{X}-p) \neq\tau(X)$. Since (from C1) $Y \subseteq B_{X}-p \cup Y
\subseteq\tau(X)$ ,then from C22 $\tau(B_{X}-p \cup Y) =\tau(X)$. Let
$\emptyset\subset C \subseteq Y$ be a minimal set such that $\tau(B_{X}-p \cup
C)= \tau(X)$. Consider some element $q \in C$, and let $Z=B_{X}-p \cup C-q$.
From minimality of $C$ follows that $\tau(Z) \neq\tau(X)$.  Note that $\tau(Z
\cup p)=\tau(X)$, that follows from $B_{X} \subseteq Z \cup p \subseteq
\tau(X)$ and C22. Thus, $\tau(Z) \neq\tau(Z \cup p)$, and from the Lemma $p
\notin\tau(Z)$. By analogy, since  $\tau(Z \cup q)= \tau(B_{X}-p \cup C)=
\tau(X)$, $q \notin\tau(Z)$.  Now, $p,q \notin\tau(Z)$, but $p \in\tau(Z \cup
q)=\tau(X)$ and $q \in\tau(Z \cup p)=\tau(X)$, contradicting the anti-exchange
axiom. Consequently $B_{X}\subseteq Y$.
\end{proof}

Regarding this theorem we can ask two questions:

1. If the same theorem is right both for closure spaces (that for the case turn to be convex geometries) and for violator spaces, is each uniquely generated violator space is a closure space (convex geometry)?

2. Is the theorem right also for weaker case of violator space?

Both answers are negative. The Example \ref{ex1} shows that there is a uniquely generated violator space that does not satisfy the property C2, and so it is not a closure space.

The following example shows that for convex spaces the theorem is not correct.

\begin{example}
\label{exms}
Let $H=\{1,2,3,4\}$. Define $\tau(X)=X$ for each $X \subseteq H$ except
$\tau(\{1,2\})=\{1,2,3\}$, and $\tau(\{1,3\})=\tau(\{1,2,3\})=\tau(\{1,3,4\})=\tau(\{1,2,3,4\})=\{1,2,3,4\}$. It's easy to check that the space
$(H,\tau)$ is uniquely generated and satisfies the properties C1 and convexity. Let $X =\{1\}$. Then there are $p=2 \in \tau(\{1,3\})=\{1,2,3,4\},
q=3 \in \tau(\{1,2\})=\{1,2,3\}$. Thus the operator $\tau$ does not satisfy the anti-exchange property.
\end{example}

\section{Extreme points}

In the section we focus on an important geometric aspect of convex sets,
namely, on the role of extreme points.

We call an element $x$ of a subset $A \subseteq E$ an \textit{extreme point of
A } if $x \notin\tau(A - x)$. The set of extreme points of $X$ is denoted by $ex(X)$.

For violator spaces from Lemma \ref{Lemma} it follows $x \notin%
\tau(A - x) \Leftrightarrow\tau(A) \neq\tau(A - x)$. Thus we have
\begin{proposition}
\label{expoint}
For violator spaces: 
$x \in ex(X)\Leftrightarrow\tau(A) \neq\tau(A - x)$.

For convex spaces:
$x \in ex(X) \rightarrow \tau(A) \neq\tau(A - x)$.
 \end{proposition}
 
 The statement for convex spasec follows straightforward from the proof of the Lemma \ref{Lemma}. The oppositive direction
may be not  correct,as we can see in Example \ref{exms}: 
 $3$ is not an extreme point for $\{1,2,3\}$, since $ 3 \in \tau(\{1,2\})$, but $ \tau(\{1,2\}) \neq  \tau(\{1,2,3\})$.

In this section we suppose that all generators and, in particular, bases of
every set $X$ are contained in $X$. The following proposition, connecting
between extreme point and bases were proved in \cite{VS}. We extend it to all generators.

\begin{proposition}
\label{exp} Let $(E,\tau)$ be a violator space. Then $x$ is an extreme point
of $X \subseteq E$ if and only if $x$ is contained in every generator (and so
in every basis) of $X$.
\end{proposition}

\begin{proof}
If $x$ is not an extreme point, then $\tau(X) = \tau(X - x)$. Then there is a
generator of $X-x$ not containing $x$.

Conversely, $x$ is an extreme point,and there is some generator $B \subseteq X$ not containing $x$,
then $B \subseteq X-x \subseteq X$. From convexity $\tau(X-x) = \tau(X)$. Contradiction.
\end{proof}

\begin{corollary}
If $(E,\tau)$ is a violator space, then
\begin{equation}
ex(X) = \bigcap\{B \subseteq X: \tau(B)= \tau(X)\}\label{expb},%
\end{equation}
if  $(E,\tau)$ is a convex space, then
\begin{equation}
ex(X) \subseteq \bigcap\{B \subseteq X: \tau(B)= \tau(X)\}\label{expc},%
\end{equation}
\end{corollary}

Since each basis of $X$ is a basis of $\tau(X)$, for violator spaces we have

$ex(\tau(X)) = \bigcap\{B \subseteq X: \tau(B)= \tau(X)\}= ex(X)$.

In particular, $ex(\tau(X)) \subseteq X$.

Now we prove that uniquely generated violator spaces satisfy the Krein-Milman property.

\begin{theorem}
\label{tex}
Let $(E,\tau)$ be a violator space. Then $(E,\tau)$ is uniquely generated if
and only if for every set $X \subseteq E$ , $\tau(X) = \tau(ex(X))$.
\end{theorem}

\begin{proof}
$\Rightarrow$ Let $B_{X}$ be a minimal set contained in $X $ such that
$\tau(B_{X})=\tau(X)$, i.e., $B_{X}$ be a basis of $X$. Prove $ex(X)= B_{X}$.
From Preposition \ref{exp} follows that $ex(X) \subseteq B_{X}$. Suppose that
there are $x \in B_{X}$ that is not an extreme point. Then $\tau(X-x) =
\tau(X)$. Unique generation implies (Proposition \ref{U1}) $B_{X} \subseteq
X-x$, contradiction. (or - from Proposition \ref{U1} $B_{X} \subseteq
\bigcap\{B \subseteq X: \tau(B)= \tau(X)\}= ex(X)$.) Then, $\tau(X)
=\tau(B_{X})= \tau(ex(X))$.

$\Leftarrow$ If $\tau(X) = \tau(ex(X))$, then $ex(X)$ is a basis, and from
(\ref{expb}) follows that $ex(X)$ is an unique minimal basis of $X$.
\end{proof}

The theorem is not valid for convex spaces as we can see from  Example \ref{exms}:
 $3$ is not an extreme point for $\{1,2,3\}$, but $3 \in \{1,3\}$ - the basis of $\{1,2,3\}$.

 In some works \cite{VS} an element $x$ of a subset $A \subseteq E$ is defined as an extreme point of
A if and only if $\tau(A) \neq \tau(A - x)$. We denote the set of such points $EX(A)$. For violator spaces this definition is equivalent to the original definition, i.e., $ex(A)=EX(A)$,  but for convex spaces $ex(A) \subseteq EX(A)$ that follows immediately from Preposition \ref{expoint}.

The second definition of extreme points ($EX()$)
allows to prove the Krein-Milman property for convex spaces.
\begin{theorem}
\label{tex2}
Let $(E,\tau)$ be an convex space. Then $(E,\tau)$ is uniquely generated if
and only if for every set $X \subseteq E$ , $\tau(X) = \tau(EX(X))$.
\end{theorem}

The proof of theorem \ref{tex} is correct also for theorem \ref{tex2}, since it uses only the second definition of extreme points.

\section{Hypercube partitions}

The section is based on the results proved in \cite{Clarkson}. Our approach (relation to closure operator) allows to give more simple proofs.
We also extend the result to convex spaces.

Let $E = {x_{1}, x_{2},..., x_{d}}$. Define a graph $H(E)$ as follows: the vertices are
the finite subsets of $E$, two vertices $A$ and $B$ are adjacent if and only if they differ in exactly one element, i.e., $B = A \cup x$ for some          $x \in E$.  Then $H(E)$ is \textit{the hypercube} on $E$ of dimension $d$. The hypercube can be equivalently defined as the graph on $\{0, 1\}^{d}$ in which two vertices form an edge if and only if they differ in exactly one position.

For the sets $A \subseteq B \subseteq E$, we define $[A,B] := \{C \subseteq E|A \subseteq C \subseteq B\}$
and call any such $[A,B]$ an \textit{interval}.  A \textit{hypercube partition} is a partition of $H(E)$ into disjoint intervals.

Let $(E,\tau)$ be a space. We call two sets $X$ and $Y$ \textit{equivalent} if $\tau(X)=\tau(Y)$, and let $\mathcal{P}$ be a partition of $H(E)$ into equivalence classes w.r.t.(with regard to) this relation.

\begin{proposition}
\label{HP}
Let $(E,\tau)$ be an uniquely generated violator space. Then $\mathcal{P}$  is a hypercube partition of $H(E)$.
\end{proposition}

\begin{proof}
We show that each equivalent class is an interval. Let $[A]=\{X \subseteq E | \tau(X)=\tau(A)\}$. Prove that $[A]=[B_{A}, \tau(A)]$, where $B_{A}$ is a unique basis of $A$.

The inclusion $[B_{A}, \tau(A)] \subseteq [A]$ follows from C3 ($\tau(\tau(A))=\tau(A)$) and convexity.
Let $B \subseteq E$ and $\tau(B)=\tau(A)$. C1 implies $B \subseteq \tau(B)=\tau(A)$. Proposition \ref{U1} implies $B_{A} \subseteq B$. So
$B \in [B_{A}, \tau(A)]$.

\end{proof}

\begin{corollary}
Elements of equivalent classes of unique generated violator spaces are closed under intersection and under union.
\end{corollary}

This property may be obtained independently. Such closeness under intersection for unique generated convex spaces was already proved  (Proposition \ref {UQP}) in previous section.
\begin{proposition}
 Let$(E,\tau)$  be a violator space. Then
for every $X,Y\subseteq E$
\begin{equation}
\tau(X)=\tau(Y)\Rightarrow\tau(X \cup Y)=\tau(X)=\tau(Y) \label{Union}
\end{equation}

\end{proposition}

\begin{proof}
From C1 follows that $X \subseteq \tau(X)$ and $Y \subseteq \tau(Y)=\tau(X)$. Then $ X \subseteq X \cup Y \subseteq \tau(X)$, that implies (by C22)
$\tau(X \cup Y)=\tau(X)$.
\end{proof}

This property is not necessary for convex spaces.

\begin{example}
Let $H=\{1,2,3,4\}$. Define $\tau(X)=X$ for each $X \subseteq H$ except
$\tau(\{1\})=\tau(\{1,2\})=\tau(\{1,3\})=\{1,2,3\}$, and $\tau(\{1,2,3\})=\{1,2,3,4\}$. It's easy to check that the space
$(H,\tau)$ is uniquely generated and satisfies the properties C1 and convexity. But $\tau(\{1,2\} \cup \{1,3\}) \neq \tau\{1,2\}$.
It easy to see that $[\{1\}]$ is not an interval.
\end{example}

The property (\ref{Union}) is equivalent to existence of unique maximal generator for each set.
Denote by $G_{Max}(A)$ - unique maximal generator of $A$.

\begin{proposition}
Let $(E,\tau)$ be a uniquely generated convex space with unique maximal generators. Then $\mathcal{P}$  is a hypercube partition of $H(E)$.
\end{proposition}

\begin{proof}
We show that each equivalent class is an interval. Let $[A]=\{X \subseteq E | \tau(X)=\tau(A)\}$. Prove that $[A]=[B_{A}, G_{Max}(A)]$, where $B_{A}$ is a unique basis of $A$ and $G_{Max}(A)$ is a unique maximal generator of $A$..

The inclusion $[B_{A}, G_{Max}(A)] \subseteq [A]$ follows from convexity.
Let $B \subseteq E$ and $\tau(B)=\tau(A)$. Existence of unique maximal generator implies $B \subseteq G_{Max}(A)$. Proposition \ref{U1} implies $B_{A} \subseteq B$. So $B \in [B_{A}, \tau(A)]$.

\end{proof}

Example \ref{exms} shows that there exist uniquely generated convex spaces with unique maximal generators that are not violator space.
So a hypercube partition may be obtained as a partition of $H(E)$ into equivalence classes not only by a violator space. Moreover, the same partition may be obtained by different type of spaces.

\begin{theorem}\cite{Clarkson}
Every hypercube partition P is a partition of $H(E)$ into equivalence classes of an uniquely generated violator space.
\end{theorem}

\begin{proof}
For each $X \subseteq E$ there is an interval $[A,B]$ containing $X$. Then define $\tau(X)=B$. Prove that $(E,\tau)$ is a violator space, i.e., $\tau$ satisfies the closure axioms:C1,C22. C1 follows from $X \subseteq B=\tau(X)$.

If $X \subseteq Y  \subseteq \tau(X)$, then $Y$ belongs to interval $[A,B]$ containing $X$, and so $\tau(Y)=\tau(X)$.
To prove unique generation note that $\tau(X)=\tau(Y)$ means $X,Y \in [A,B]$. Then $\tau(X \cap Y) = \tau(X)$, and from Proposition 2.5 immediately follows that the violator space is uniquely generated.
\end{proof}

\end{document}